\documentclass[11pt,reqno]{amsart}

\usepackage{amssymb}
\usepackage{amscd}
\usepackage{amsfonts}
\usepackage{mathrsfs}
\usepackage{setspace}
\usepackage{version}
\usepackage{mathtools}
\usepackage{enumerate}
\usepackage[english]{babel}
\usepackage{xcolor}
\usepackage[colorlinks=true,linkcolor=blue]{hyperref}

\newtheorem{theorem}{Theorem}[section]
\newtheorem{lemma}[theorem]{Lemma}
\newtheorem{proposition}[theorem]{Proposition}
\newtheorem{corollary}[theorem]{Corollary}
\theoremstyle{definition}
\newtheorem{definition}[theorem]{Definition}

\newtheorem{example}[theorem]{Example}

\theoremstyle{remark}
\newtheorem{remark}[theorem]{Remark}
\numberwithin{equation}{section}

\newcommand{\field}[1]{\mathbb{#1}}

\newcommand{\R}{\field{R}}
\newcommand{\N}{\field{N}}

\newcommand{\Lip}{\mathop{\rm{Lip}}}

\newcommand{\de}{\delta}
\newcommand{\ep}{\varepsilon}

\newcommand{\la}{\lambda}

\newcommand{\si}{\sigma}

\newcommand{\Om}{\Omega}

\newcommand{\BUC}{\text{\rm{BUC}}}

\newcommand{\mnen}{}

\begin{document}

	\title[Convex monotone semigroups on lattices of continuous functions]{Convex monotone semigroups on lattices of continuous functions}
	
	\author{Robert Denk}
	\address{Department of Mathematics and Statistics, University of Konstanz, Germany}
	\email{robert.denk@uni-konstanz.de}
	
	\author{Michael Kupper}
	\address{Department of Mathematics and Statistics, University of Konstanz, Germany}
	\email{kupper@uni-konstanz.de}
	
	\author{Max Nendel}
	\address{Center for Mathematical Economics, Bielefeld University, Germany}
	\email{max.nendel@uni-bielefeld.de}

	\date{\today}

	\thanks{Financial support through the German Research Foundation via CRC 1283 is gratefully acknowledged by the third author.\ We thank Daniel Bartl, Jonas Blessing, Liming Yin and Jos\'e Miguel Zapata Garc\'ia  for helpful discussions and comments.}
	
	\subjclass[2010]{}

	\begin{abstract}
	We consider convex monotone $C_0$-semigroups on a Banach lattice, which is assumed to be a Riesz subspace of a $\sigma$-Dedekind complete Banach lattice. 
	Typical examples include the space of all bounded uniformly continuous functions and the space of all continuous functions vanishing at infinity. 
	We show that the domain of the classical generator of a convex semigroup is typically not invariant. Therefore, we propose alternative versions for the domain, 
	such as the monotone domain and the Lipschitz set, for which we prove invariance under the semigroup. As a main result, we obtain the uniqueness of the semigroup in terms of an extended version of the generator. 
	The results are illustrated with several examples related to Hamilton-Jacobi-Bellman equations, including nonlinear versions of the shift semigroup and the heat equation.
	In particular, we determine their symmetric Lipschitz sets, which are invariant and allow to \mnen{understand} the generators in a weak sense.
	

\end{abstract}

	\maketitle
	
	\setcounter{tocdepth}{1}

\section{Introduction}

The topic of model uncertainty or ambiguity in the fields of Mathematical Economics and Mathematical Finance has been extensively studied in the past decades. Hereby, a particular focus has been put on parameter uncertainty of stochastic processes describing the evolution of an underling asset. Examples include a Brownian motion with drift uncertainty (cf.~Coquet et al.~\cite{MR1906435}) or volatility uncertainty (cf.\ Peng~\cite{PengG},\cite{MR2474349}), a Black-Scholes model with volatility uncertainty (cf.\ Avellaneda et al.\ \cite{doi:10.1080/13504869500000005}, Epstein and Ji \cite{Epji}, Vorbrink \cite{MR3250653}), and L\'evy processes with uncertainty in the L\'evy triplet (cf.\ Hu and Peng \cite{PengHu}, Neufeld and Nutz \cite{NutzNeuf}, Hollender \cite{H2016}, K\"uhn \cite{MR3941868}). 
In the case of a Brownian motion with uncertain volatility within an interval $[\underline \si, \overline \si]$ for $0\leq \underline \si\leq \overline \si$, 
this leads to the  equation
\begin{equation}\label{eq.G-heat}
  \partial_t u(t,x)=\sup_{\sigma\in [\underline \si, \overline \si]} \frac{\sigma^2}{2}\partial_{xx}u(t,x)\quad \text{for }t\geq 0\text{ and }x\in \R.
\end{equation}
The latter is referred to as the $G$-heat equation, and their solutions (for different initial values) can be represented by means of the so-called $G$-expectation, cf.\ Peng~\cite{PengG,MR2474349}.

Equation \eqref{eq.G-heat} falls into the class of Hamilton-Jacobi-Bellman (HJB) equations, which closely relate to (stochastic) optimal control problems. On a meta (and structurally very reduced) level, a control problem consists of a nonempty control set $\Lambda$, and a family of semigroups $(S_\lambda)_{\lambda\in \Lambda}$, where $S_\lambda$ relates to the value of a cost functional under the static control $\lambda\in \Lambda$. Allowing for a dynamic state-dependent choice 
from the control set $\Lambda$, leads to a convex semigroup $S$, which on an abstract level is given as a (viscosity) solution to an HJB equation of the form
 \begin{equation}\label{abstractHJB}
  \partial_t u =  \sup_{\lambda\in \Lambda} A_\lambda u,
 \end{equation}
 where $A_\lambda$ is the generator of the affine linear semigroup $S_\lambda$ for all $\lambda\in \Lambda$. Inspired by a construction of Nisio \cite{MR0451420}, such equations have been studied using a semigroup-theoretic framework for spaces of continuous functions by Denk et al.~\cite{dkn2} and Nendel and R\"ockner \cite{roecknen}. Choosing $A_\lambda:=\tfrac{\lambda^2}{2}\partial_{xx}$ for $\lambda\in \Lambda:=[\underline \si, \overline \si]$, the $G$-heat equation \eqref{eq.G-heat} is a particular instance of \eqref{abstractHJB}. The related control problem is the one of optimally choosing a volatility from the control set $[\underline \si, \overline \si]$. We refer to Denis et al.\ \cite{MR2754968} for a detailed illustration of this relation. In the context of optimal control theory, the uniqueness and regularity of solutions to Hamilton-Jacobi-Bellman equations are fundamental in order to come up with verification theorems; ensuring that the solution to the HJB equation is in fact the value function of an optimal control problem, cf.\ Fleming and Soner \cite{MR2179357}, Pham \cite{MR2533355}, and Yong and Zhou \cite{MR1696772}. In an even broader sense, the $G$-heat equation and HJB equations of the form \eqref{abstractHJB} 
 are examples for convex differential equations and the related value functions (their solutions) form a convex semigroup on suitable spaces of continuous functions, where the semigroup property is the abstract analogon of the dynamic programming principle. We refer to Denk et al.~\cite{dkn2} and Nendel and R\"ockner \cite{roecknen} for more details on this relation.

 One classical approach to treat nonlinear equations uses the theory of maximal monotone or m-accretive operators, cf.\ Barbu \cite{Barbu10}, B\'enilan and Crandall \cite{Benilan-Crandall91}, Br\'ezis \cite{Brezis71}, Evans \cite{Evans87}, Kato \cite{Kato67}, and the references therein. To show that an accretive operator is m-accretive, one has to prove that $1+h A$ is surjective for  $h>0$, and in many cases it is quite delicate to verify this condition, see Example \ref{1.5} below. Moreover, it is known that m-accretive operators lead to the existence of a mild solution, but the existence of strong solutions is only known under additional assumptions on the underlying Banach space, including reflexivity, see \cite[Section~4.1]{Barbu10}. In terms of nonlinear semigroups, this means that even if the initial value is smooth, the solution (the semigroup applied to the initial value) does not belong to the domain of the operator for positive time, so the domain of the operator is not invariant under the semigroup, see \cite[Section~4]{Crandall-Liggett71} or Example~\ref{4.3} below.
Therefore, in the context of HJB equations, one typically considers a more general solution concept, so-called viscosity solutions, cf.\ Crandall et al.\ \cite{Crandall-Ishii-Lions92}, Crandall and Lions \cite{Crandall-Lions83}, and the discussion in Evans \cite[Section~4]{Evans87}.


In this paper, we study  convex monotone semigroups on spaces of continuous functions and construct invariant domains with a particular interest in the regularity and  uniqueness of the solution. The main object and the starting point of our investigation is a convex $C_0$-semigroup $S=(S(t))_{t\geq 0}$ on a Banach lattice $X$ which is a Riesz subspace of some Dedekind $\si$-complete Riesz space $\overline{X}$. Typical examples for $X$ are the space $\BUC$ of all bounded uniformly continuous functions, the space $C_0$ of all continuous functions vanishing at infinity, or spaces of uniformly continuous functions with certain growth at infinity. We focus on monotone semigroups that are continuous from above, meaning that $S(t)x_n\downarrow 0$ for all $t\geq 0$, whenever $x_n\downarrow 0$. This additional continuity property allows to extend the semigroup to the set $X_\de$ of all $x\in \overline{X}$ for which there exists a sequence $(x_n)_{n\in\mathbb{N}}$ in $X$ such that $x_n\downarrow x$. Likewise, the generator $A$ of the semigroup extends to the so-called monotone generator $A_\delta$, \mnen{whose domain} is defined as the set of all $x\in X$ such that, for every sequence $(h_n)_{n\in\mathbb{N}}$ in $(0,\infty)$ with $h_n\downarrow 0$, there exists an approximating sequence $(y_n)_{n\in\mathbb{N}}$ in $X$ such that 
\[
\bigg\|\frac{S(h_n)x-x}{h_n}-y_n\bigg\|\to 0\quad \text{and}\quad y_n\downarrow y=:A_\delta x.
\]
The main results in Section~\ref{sec:semigroupmonotone} and Section~\ref{sec:uniqueness} state that a convex monotone $C_0$-semigroup leaves the domain $D(A_\delta)$ of its monotone generator invariant, and that the semigroup is uniquely determined by $A_\delta$ on $D(A_\delta)$. We also study even weaker forms of domains requiring only the local Lipschitz continuity of the map $t\mapsto S(t)x$, or, in other words, a weak Sobolev regularity of the map $t\mapsto S(t)x$, i.e., for every continuous linear functional $\mu$, the map $\big(t\mapsto \mu S(t)x\big)\in W_{\rm loc}^{1,\infty}\big([0,\infty)\big)$. These domains are shown to be invariant as well, and we discuss their relation to one another.

In Section~\ref{sec:examples}, we consider the
example of  the uncertain shift semigroup \mnen{on the space of $\BUC$ of all bounded uniformly continuous functions}, which corresponds to the fully nonlinear PDE
\begin{equation}\label{1-1}
\partial_t u(t,x)= |\partial_x u(t,x)|,\quad u(0,\cdot)=f.
\end{equation}
Here, the nonlinear generator is given by  $Au =|\partial_x u|$ \mnen{for sufficiently regular $u\in \BUC$}. In that case, it holds $\BUC^1\subset D(A_\delta)\subset W^{1,\infty}$ and $ W^{1,\infty}$ is invariant under the corresponding semigroup. Note that
\eqref{1-1} is a special case of the Hamilton-Jacobi PDE, where, under appropriate conditions on the nonlinearity, the
viscosity solution is given by the Hopf-Lax formula, see, e.g., \cite[Section~3.3]{Evans10}, \cite[Section~11.1]{Lions82}. Similarly, for the second-order differential operator $Au=\frac{1}{2}\max\{\underline{\sigma}^2\partial_{xx}u,\overline{\sigma}^2\partial_{xx}u\}$ with $0\le \underline{\sigma}\le \overline{\sigma}$, we derive that $W^{2,\infty}$ is invariant under the respective semigroup $S$, which corresponds to the $G$-heat equation.\ \mnen{Moreover, we show that the equality $\lim_{h\downarrow 0}\frac{S(h)u-u}{h}=\frac{1}{2}\max\{\underline{\sigma}^2\partial_{xx}u,\overline{\sigma}^2\partial_{xx}u\}$ holds on $W^{2,\infty}$ in a pointwise sense almost everywhere.}\ We remark that in the parabolic situation $\underline{\sigma}>0$
many results on the solvability of this second-order fully nonlinear equation in Sobolev and H\"older spaces were
obtained by Krylov, see \cite[Chapters 12 and 13]{Krylov18b}.

\section{Setup and notation}\label{sec:setup}
Throughout this article, we assume that $X$ is a real Banach lattice which is a Riesz subspace of a Dedekind $\sigma$-complete Riesz space $\overline X$. A typical example is the space $\BUC$ as a subspace of the space  $\mathcal L^\infty$ of all bounded measurable functions. We denote by $X'$ the dual space of $X$, i.e., the space of all continuous linear functionals $X\to \R$. For a sequence $(x_n)_{n\in\mathbb{N}}$ in $ X$, we write $x_n\downarrow x$ if $(x_n)_{n\in\mathbb{N}}$ is decreasing, bounded from below, and $x=\inf_n x_n\in \overline X$.
We define
\[
X_\delta:=\left\{x\in \overline X\colon x_n\downarrow x\mbox{ for some sequence }(x_n)_{n\in\mathbb{N}}\mbox{ in }X\right\}.
\]
Let $M$ be the space of all positive linear functionals $\mu\colon X\to \mathbb{R}$ which are continuous from above, i.e.,~$\mu x_n\downarrow 0$ for every sequence $(x_n)_{n\in\mathbb{N}}$ in $X$ such that $x_n\downarrow 0$. Every $\mu\in M$ has a unique extension $\mu\colon X_\delta\to\mathbb{R}$ which is continuous from above, i.e.,~$\mu x_n\downarrow \mu x$ for every sequence $(x_n)_{n\in\mathbb{N}}$ in $X_\delta$ such that $x_n\downarrow x\in X_\delta$, see e.g.~\cite[Lemma 3.9]{denk2018kolmogorov}. We assume that the set $M$ separates the points of $X_\delta$, i.e.,~for every $x,y\in X_\delta$ with $x\neq y$ there exists some $\mu\in M$ with $\mu x\neq \mu y$. For an operator $S\colon X\to X$, we define
\[
 \|S\|_r:=\sup_{x\in B(0,r)}\|S x\|
\]
for all $r>0$, where $B(x_0,r):=\{x\in X\colon \|x-x_0\|\leq r\}$ for $x_0\in X$. We say that an operator $S\colon X\to X$ is \textit{convex} if $S \big(\lambda x+(1-\lambda) y\big)\leq \lambda S x+ (1-\lambda) Sy$ for all $\lambda \in[0,1]$, \textit{positive homogeneous} if $S (\lambda x)= \lambda S x$ for all $\lambda>0$, \textit{sublinear} if $S$ is convex and positive homogeneous, \textit{monotone} if $x\le y$ implies $Sx\le Sy$ for all $x,y\in X$, and \textit{bounded} if $\|S\|_r<\infty$ for all $r>0$.
For $x\in X$, we define the convex operator $S_x\colon X\to X$ by
\[
S_x y:=S(x+y)-Sx.
\]
\begin{definition}
 A family $S=\big(S(t)\big)_{t\geq 0}$ of bounded operators $X\to X$ 
 is called a \emph{$C_0$-semigroup} on $X$ if 
 \begin{enumerate}
  \item[(S1)] $S(0)x=x$ for all $x\in X$,
  \item[(S2)] $S(t+s)x=S(t)S(s)x$ for all $x\in X$ and $s,t\in[0,\infty)$,
  \item[(S3)] $S(t)x\to x$ as $t\downarrow 0$ for all $x\in X$.
 \end{enumerate}
We say that  $S$ is \textit{monotone}, \textit{convex}, or \textit{sublinear} if $S(t)$ is monotone, convex, or sublinear for all $t\ge 0$, respectively.  
 \end{definition}
 We conclude with a notion of continuity, which we will require on several occasions.
 \begin{definition}
 	A monotone $C_0$-semigroup  $S$ is called \textit{continuous from above} if $S(t)x_n\downarrow S(t)0$ for all $t\in[0,\infty)$ and every sequence $(x_n)_{n\in\mathbb{N}}$ in $X$ with $x_n\downarrow 0$.
 \end{definition}

\section{Invariant domains}\label{sec:semigroupmonotone}
 In this section, we discuss the invariance of various notions of generators and domains. 
Throughout, let $S$ be a convex $C_0$-semigroup on $X$.
In contrast to \cite{dkn3}, where the Banach lattice $X$ is Dedekind $\si$-complete with order continuous norm, the domain
\[
D(A):=\bigg\{x\in X\colon \frac{S(h)x-x}{h}\text{ is convergent in $X$ for } h\downarrow 0\bigg\}
\]
is in general not invariant under the semigroup. For instance, for the uncertain semigroup $(S(t))_{t\in[0,\infty)}$ in Section \ref{sec:ushift}, there exists some $x\in D(A)$ such that $S(t)x\not\in D(A)$ for some $t\in(0,\infty)$. We therefore introduce the following modified versions of the domain.

\begin{definition}\label{def:dom2}
	The domain $D(A_\delta)$ of the \textit{monotone generator} $A_\delta$ of $S$ is defined  as  the set of all $x\in X$ such that, for every $(h_n)_{n\in\mathbb{N}}$ in $(0,\infty)$ with $h_n\downarrow 0$, there exists a sequence $(A_n x)_{n\in\mathbb{N}}$ in $X$ and some $y\in X_\delta$ such that
	\begin{equation}\label{domweak}
	\bigg\| \frac{S(h_n)x-x}{h_n}-A_n x \bigg\|\to 0\quad\mbox{and}\quad A_n x\downarrow y.
	\end{equation}
	We define the monotone generator $A_\delta\colon D(A_\delta)\subset X\to X_\delta$  by $A_
	\delta x:=y$ for $x\in D(A_\delta)$, where $y$ is the limit in \eqref{domweak}, which is uniquely determined by Lemma \ref{lemma:welldef}.
\end{definition}

\begin{definition}
	The \emph{Lipschitz set} of the semigroup $S$ is defined as
	\begin{equation}\label{dom:LIp}
	D_L:=\bigg\{x\in X\colon \sup_{h\in (0,h_0]}\bigg\|\frac{S(h)x-x}{h}\bigg\|<\infty\quad  \mbox{for some }h_0>0\bigg\}.
	\end{equation}
	We further define the \emph{symmetric Lipschitz set} of the semigroup $S$ by
	\[
	 D_L^s:=\big\{x\in X\colon x,-x\in D_L \big\}.
	\]
\end{definition}

Let $W^{1,\infty}_{\rm loc}\big([0,\infty)\big)$ denote the space of all functions $f\colon [0,\infty)\to \R$ in $L^\infty_{\rm loc}\big([0,\infty)\big)$ with weak derivative $f'\in L^\infty_{\rm loc}\big([0,\infty)\big)$. Recall that $ W^{1,\infty}_{\rm loc}\big([0,\infty)\big)$ coincides with space of all locally Lipschitz continuous functions. The following observation is one of the basic ingredients in the proof of Section \ref{sec:uniqueness}, below.
\begin{remark}
Let $x\in X$. Then, $x\in D_L$ if and only if $$\big(t\mapsto \mu S(t)x\big)\in W^{1,\infty}_{\rm loc}\big([0,\infty)\big)\quad \text{for all }\mu\in X'.$$ In fact, by Proposition \ref{domainlip}, the map $[0,\infty)\to X, \;t\mapsto \mu S(t)x$ is locally Lipschitz for every $x\in D_L$ and $\mu\in X'$, which proves one direction of the equivalence. Now, assume that $\big(t\mapsto \mu S(t)x\big)\in W^{1,\infty}_{\rm loc}\big([0,\infty)\big)$ for all $\mu\in X'$. Then, for every $\mu\in X'$,
 \[
  \sup_{h\in (0,1]} \bigg|\mu \bigg(\frac{S(h)x-x}{h}\bigg)\bigg|<\infty.
 \]
 By the Banach-Steinhaus theorem, it follows that $x\in D_L$. If $\sup_{t\geq 0}\|S(t)\|_r<\infty$ for all $r\geq 0$, as, for example, in Section \ref{sec:ushift} and Section \ref{sec:G}, we obtain that $x\in D_L$ if and only if $$\big(t\mapsto \mu S(t)x\big)\in W^{1,\infty}\big([0,\infty)\big)\quad \text{for all }\mu\in X'.$$
\end{remark}

We say that the norm $\|\cdot \|$ on $X$ is \textit{$\sigma$-order continuous} if $\lim_{n\to\infty}\|x_n\|=0$ for every decreasing sequence $(x_n)_{n\in\mathbb{N}}$ with $\inf_{n\in \N}x_n=0$. The prime example for a Banach lattice with $\sigma$-order continuous norm is the closure $C_0$ w.r.t.\ supremum norm $\|\cdot \|_\infty$ of the space $C_c$ of all continuous functions $\Om \to \R$ with compact support, where $\Om$ is a locally compact metric space. Moreover, we say that the norm $\|\cdot \|$ on $X$ is \textit{order continuous} if, for every net $(x_\alpha)_\alpha$ with $x_\alpha \downarrow 0$, we have $\|x_\alpha\|\to 0$. Notice that order continuity of the norm is, for example, implied by separability of $X$ together with \textit{Dedekind $\si$-completeness} of $X$, 
cf.\ \cite[Exercise 2.4.1]{MR1128093} or \cite[Corollary to Theorem II.5.14]{MR0423039}. Typical examples for Banach lattices with order continuous norm are the spaces $L^p(\mu)$ for $p\in [1,\infty)$ and an arbitrary measure $\mu$, the space $c_0$ of all sequences vanishing at infinity, and Orlicz spaces. We would like to point out that, due to its strong implications, we avoid order continuity of the norm in the present paper.\ A detailed study of convex semigroups on Banach lattices with order continuous norm can be found in \cite{dkn3}.

We have the following relations between the domains and generators.
\begin{lemma}\label{lem:domain}
	One has $D(A)\subset D(A_\delta)\subset D_L$, and $A_\delta|_{D(A)}=A$. If the norm $\|\cdot \|$ on $X$ is $\sigma$-order continuous, then $x\in D(A_\delta)$ with $A_\delta x\in X$ implies $x\in D(A)$ and $A_\delta x=Ax$. 
	If  $X$  is $\sigma$-Dedekind complete with  $\sigma$-order continuous norm, then $A_\delta=A$.
\end{lemma}
\begin{proof}
	We first assume that $x\in D(A)$. Then, for every  $h_n\downarrow 0$ and $A_n x:= A x$ for all $n\in\mathbb{N}$, one has
	\[
	\bigg\| \frac{S(h_n)x-x}{h_n}-A_n x \bigg\|\to 0,
	\]
	which shows that $x\in D(A_\delta)$ with $A_\delta x=A x$.
	
	We next assume that $x\in D(A_\delta)$.
	Then, there exists some $h_0>0$ such that
	\[
	\sup_{h\in(0,h_0]}\bigg\|\frac{S(h)x-x}{h}\bigg\|<\infty.
	\]
	Otherwise, there exists a sequence $h_n\downarrow 0$ such that $\big\|\tfrac{S(h_n)x-x}{h_n}\big\|\ge n$ for all $n$. Since $x\in D(A_\delta)$ there exists a bounded decreasing sequence $(A_n x)_{n\in\mathbb{N}}$ in $X$ such that $A_n x\downarrow A_\delta x$ and
	\[\bigg\| \frac{S(h_n)x-x}{h_n}-A_n x \bigg\|\to 0.\] But then,
	\[\sup_{n\in\mathbb{N}}\bigg\|\frac{S(h_n)x-x}{h_n}\bigg\|\le\sup_{n\in\mathbb{N}} \bigg\|\frac{S(h_n)x-x}{h_n}-A_n x\bigg\|+\sup_{n\in\mathbb{N}}\|A_n x\|<\infty,
	\]
	which is a contradiction. This shows that $x\in D_L$. 
	
	If the norm $\|\cdot \|$ on $X$ is $\sigma$-order continuous and $x\in D(A_\delta)$ with $A_\delta x\in X$, then $\|A_n x- A_\delta x\|\to 0$, so that $\tfrac{S(h_n)x-x}{h_n}\to A_\delta x$. If, in addition, $X$ is $\sigma$-Dedekind complete, then $A_\delta x\in X$ for all $x\in D(A_\delta)$, which shows that $A_\delta=A$.
\end{proof}	
For every $x\in X$ and $y\in X_\delta$, the directional derivative is defined as
\[
S^\prime_+(t,x)y=\inf_{h>0} \frac{S(t)(x+h  y)-S(t)x}{h }\in X_\delta.
\]
For further details on the directional derivative we refer to  Appendix \ref{append:direcder}.
The following main result of this subsection provides invariance for $D_L$ and $D(A_\delta)$, and states
regularity properties in the time variable.

\begin{theorem}\label{generatormain2}
	For every $x\in D_L$, one has
	\begin{itemize}
		\item[(i)] $S(t)x\in D_L$ for all $t\in[0,\infty)$,
		\item[(ii)] for every $\mu\in M$ there is a locally bounded measurable function $f_\mu\colon [0,\infty)\to \R$ with $\mu S(t)x=\mu x+\int_0^t f_\mu(s)\,ds$ for all $x\in D(A_\delta)$ and $t\ge 0$.
	\end{itemize}
	For every $x\in D(A)$, it holds
	\begin{itemize}
		\item[(iii)] $S(t)x\in D(A_\delta)$ for all $t\ge 0$ with $A_\delta S(t)x=S^\prime_+(t,x)A_\delta x$,
		
		\item[(iv)] $\mu S(t)x=\mu x+\int_0^t \mu S^\prime_+(s,x)A_\delta x\,ds$ for every $\mu\in M$ and all $t\ge 0$. In particular, $f_\mu(s)=\mu S^\prime_+(s,x)A_\delta x$ for almost every $s\in[0,\infty)$.
	\end{itemize}
	Moreover, \emph{(iii)} and \emph{(iv)} hold for all $x\in D(A_\delta)$ if, in addition, the semigroup is monotone and  continuous from above.
\end{theorem}

\begin{proof}
	(i) Fix $t\ge 0$. By Corollary \ref{loclipschitz} there exist $L\ge 0$ and $r>0$ such that
	\[
	\|S(t)(y+x)-S(t)x\|\le L\|y\|
	\]
	for all $y\in B(x,r)$. Since $S(h)x\to x$ as $h\downarrow 0$,  it follows that
	\[
	\bigg\|\frac{S(h)S(t)x-S(t)x}{h}\bigg\|=\bigg\|\frac{S(t)S(h)x-S(t)x}{h}\bigg\|\le L \bigg\|\frac{S(h)x-x}{h}\bigg\|<\infty
	\]
	for all $h\in(0,h_0^\prime]$ and some $h_0^\prime>0$.

	(ii) Since $x\in D_L$, it follows from Proposition \ref{domainlip} that the map $[0,\infty)\to X,$ $t\mapsto S(t)x$ is locally Lipschitz continuous. Fix $\mu\in M$. Since $\mu$ is continuous on $X$, see e.g.~\cite[Theorem 9.6]{aliprantis2006infinite}, the map  $[0,\infty)\to \R,$ $t\mapsto \mu S(t)x$ is also locally Lipschitz continuous and is therefore in $W^{1,\infty}_{\rm loc}\big([0,\infty)\big)$ by Lebesgue's theorem.
	That is, there exists a locally bounded measurable function $f_\mu\colon [0,\infty)\to \R$ with
	$\mu S(t)x=\mu x+\int_0^tf_\mu(s)\, {\rm d}s$.
	
	(iii)  Fix $t>0$, let $(h_n)_{n\in\mathbb{N}}$ be a sequence in $(0,\infty)$ with $h_n\downarrow 0$, and $x\in D(A)$. By Corollary \ref{loclipschitz}, there exists some $L>0$ such that	
	\begin{align*}
	&\bigg\|\frac{S(t+h_n)x-S(t)x}{h_n}-\frac{S(t)(x+h_nA x)-S(t)x}{h_n}\bigg\|=\bigg\|\frac{ S(t)S(h_n)x-S(t)(x+h_nA x) }{h_n}\bigg\|\\
	&\le L \bigg\|\frac{ S(h_n)x-x - h_nA x  }{h_n}\bigg\| = L\bigg\|\frac{S(h_n)x-x}{h_n}-A x\bigg\|\to 0\quad\mbox{as }n\to\infty.
	\end{align*}
	Moreover, the  sequence
	\[
	A_n\big(S(t)x\big):=\frac{S(t)(x+h_nA x)-S(t)x}{h_n}
	\]
	is decreasing and satisfies $A_n(S(t)x)\downarrow S^\prime_+(t,x)A x$. This shows that $S(t)x\in D(A_\delta)$ with $A_\delta S(t)x=	S^\prime_+(t,x)A x$. Recall that $A x= A_\delta x$
	for all $x\in D(A)$ by Lemma \ref{lem:domain}.
	
	If in addition, $S$ is monotone, continuous from above, and $x\in D(A_\delta)$,
	then there exists
	a bounded decreasing sequence $(A_nx)_{n\in\mathbb{N}}$ in $X$ such that
	\[
	\bigg\|\frac{S(h_n)x-x}{h_n}-A_nx \bigg\|\to 0\quad\mbox{and}\quad A_n x\downarrow A_\delta x.
	\]
	By Corollary \ref{loclipschitz}, there exists some $L>0$ such that	
	\begin{align*}
	&\bigg\|\frac{S(t+h_n)x-S(t)x}{h_n}-\frac{S(t)(x+h_nA_nx)-S(t)x}{h_n}\bigg\|\le L\bigg\|\frac{S(h_n)x-x}{h_n}-A_n x\bigg\|\to 0
	\end{align*}
	as $n\to\infty$.
	By Lemma \ref{lem:Sprime0}, the sequence $(A_n S(t)x)$ given by
	\[
	A_n S(t)x:=\frac{S(t)(x+h_nA_n x)-S(t)x}{h_n}
	\]
	is decreasing and satisfies $A_n S(t)x\downarrow S^\prime_+(t,x)A_\delta x$. This shows that $S(t)x\in D(A_\delta)$ with $A_\delta S(t)x=	S^\prime_+(t,x)A_\delta x$.
	
	(iv) Since $x\in D(A_\delta)$, it follows
	from Lemma \ref{lem:domain}
	that $x\in D_L$. Fix $\mu\in M$. By (ii) one has
	\[
	\mu S(t)x=\mu x+\int_0^t f_\mu(s)\,ds
	\]
	for all $t\ge 0$. In particular, $t\mapsto \mu S(t)x$ is differentiable almost everywhere. Since $\mu$ is continuous from above it follows from the previous step (iii) that the derivative is almost everywhere given by
	\[
	f_\mu(t)=\lim_{h\downarrow 0}\frac{\mu S(t+h)x-\mu S(t)x}{h}=\mu A_\delta S(t)x=\mu S^\prime_+(t,x)A_\delta x.
	\]
	The proof is complete.
\end{proof}

For the symmetric Lipschitz set of a sublinear monotone semigroup, we have the following result.

\begin{proposition}\label{thm:domsym}
 Let $S$ be sublinear and monotone. Then, the symmetric Lipschitz set $D_L^s$ is a linear subspace of $X$. If
 \begin{equation}\label{eq:condsymlip}
  -S(s)\big(-S(t)x\big)\geq S(t)\big(-S(s)(-x)\big)\quad \text{for all }s,t\geq 0\text{ and }x\in X,
 \end{equation}
 then $S(t)x\in D_L^s$ for all $t\geq 0$ and $x\in D_L^s$.
\end{proposition}

\begin{proof}
  The sublinearity of $S$ implies that
  \[
   S(t)(x+\la y)-(x+\la y)\leq S(t)x-x+\la \big(S(t)y-y\big)
  \]
  and
  \[
   -S(t)(x+\la y)+x+\la y\leq S(t)(-x)+x+\la \big(S(t)(-y)+y\big)
  \]
  for all $x,y\in X$ and $\la>0$. Consequently,
  \[
   \|S(t)(x+\la y)-(x+\la y)\|\leq \|S(t)x-x\|+\|S(t)(-x)+x\|+\la \big(\|S(t)y-y\|+\| S(t)(-y)+y\|\big)
  \]
  for all $x,y\in X$ and $\la>0$, which shows that $x+ \la y \in D_L^s$ for all $x,y\in D_L^s$ and $\la>0$. Since $-x\in D_L^s$ for all $x\in D_L^s$, it follows that $D_L^s$ is a linear subspace of $X$.

  Now, let $x\in D_L^s$ and $t\geq 0$. Since $S(t)$ is sublinear and bounded, it is globally Lipschitz with some Lipschitz constant $L>0$ (cf.\ Lemma \ref{cor:Lip}). Therefore,
  \[
   \|S(h)S(t)x-S(t)x\|\leq L\|S(h)x-x\|,
  \]
  i.e., $S(t)x\in D_L$. It remains to show that $-S(t)x\in D_L$. First, observe that
  \[
   -S(t)x-S(h)\big(-S(t)x\big)\leq -S(t)x+S(h)S(t)x\leq S(t)\big(S(h)x-x\big)
  \]
  and, by \eqref{eq:condsymlip},
  \[
   S(h)\big(-S(t)x\big)+S(t)x\leq -S(t)\big(-S(t)(-x)\big)+S(t)x\leq S(t)\big(S(h)(-x)+x\big).
  \]
  Therefore,
  \[
   \big\|S(h)\big(-S(t)x\big)+S(t)x\big\|\leq L\Big(\|S(h)x-x\|+\big\|\big(S(h)(-x)+x\big)\big\|\Big),
  \]
  which shows that $-S(t)x\in D_L$.
\end{proof}

\begin{example}
Let $S$ be a translation invariant sublinear monotone semigroup on the space $\BUC=\BUC(G)$, where $G$ is an abelian group with a translation invariant metric $d$ such that $(G, d)$ is separable and complete. Here, \textit{translation invariant} means that
\[
\big(S(t)f(x+\cdot)\big)(0)=\big(S(t)f\big)(x) \quad \text{for all }f\in \BUC, \;x\in G\text{ and }t\ge 0.
\]
The space $\BUC$ of all bounded uniformly continuous functions $f\colon G \to \R$ is endowed with the supremum norm $\|f\|_\infty:=\sup_{x\in G}|f(x)|$.
Under mild continuity assumptions, the semigroup has a dual representation
\begin{equation}\label{dual:rep}
\big(S(t)f\big)(x)=\sup_{\mu\in\mathcal{P}_t} \int_G f(x+y)\,{\rm d}\mu_t(y)\quad\mbox{for all }f\in\BUC, \; x\in G\text{ and }t\geq 0.
\end{equation}
where $\mathcal{P}_t$ is a convex set of Borel measures on $G$ for all $t\ge 0$. For further details on dual representations we refer to \cite{denk2018kolmogorov} and, for further examples, we refer to \cite{dkn2}. Notice that, under \eqref{dual:rep},
\[
 -\big(S(t)(-f)\big)(x)=\inf_{\mu\in\mathcal{P}_t} \int_G f(x+y)\,{\rm d}\mu_t(y)\quad\mbox{for all }f\in\BUC,\; x\in G\text{ and }t\geq 0.
\]
Then, for $f\in \BUC$, $x\in G$, $\mu_t\in\mathcal{P}_t$ and $\mu_s\in\mathcal{P}_s$, it follows from \eqref{dual:rep} and Fubini's theorem that
\begin{align*}
 \int_G \big(S(t)f\big)(x+y)\,{\rm d}\mu_s(y)&\ge \int_G \int_G f(x+y+z)\,{\rm d}\mu_t(z)\,{\rm d}\mu_s(y)\\
 &=\int_G \int_G f(x+y+z)\,{\rm d}\mu_s(y)\,{\rm d}\mu_t(z)\\
 &\ge \int_G -\big(S(s)(-f)\big)(x+z)\,{\rm d}\mu_t(z).
\end{align*}
 Taking the infimum over all $\mu_s\in\mathcal{P}_t$ and supremum over all $\mu_t\in\mathcal{P}_s$ yields
 \[
  -S(s)\big(-S(t)f\big)\geq S(t)\big(-S(s)(f)\big).
 \]
 By Proposition \ref{thm:domsym}, we thus find that $D_L^s$ is $S(t)$-invariant for all $t\geq 0$.
\end{example}

\begin{remark}\label{prop:approx}
	Consider the setup of the previous example. Given $C\geq 0$ and $h_0>0$, let $D^s_L(C,h_0)$ denote the set of all $f\in D^s_L$ such that $\|S(h)f-f\|_\infty\le Ch$ and $\|S(h)(-f)+f\|_\infty\le Ch$ for all $h\in[0,h_0]$. Let $f\in D^s_L(C,h_0)$ and $\nu$ be a Borel probability measure on $G$. Then, it holds $f_\nu\in D^s_L(C,h_0)$, where $f_\nu(x):=\int_G f(x+y)\,\nu(dy)$. In fact, by a Banach space valued version of Jensen's inequality (cf.\ \cite{dkn2} or \cite{roecknen}) and the translation invariance of $S$,
	\begin{align*}
	 S(h)f_\nu-f_\nu&=S(h)\bigg(\int_G f(\,\cdot\,+y)\,{\rm d}\nu(y)\bigg)-f_\nu\leq \int_G \big(S(h)f\big)(\,\cdot\,+y)\,{\rm d}\nu(y)-f_\nu\\
	 &=\int_G \big(S(h)f\big)(\,\cdot\,+y)-f(\,\cdot\,+y)\,{\rm d}\nu(y)\leq Ch
	\end{align*}
        for all $h\geq 0$. In a similar way, it follows that
        \[
         S(h)(-f_\nu)+f_\nu\leq \int_G \big(S(h)(-f)\big)(\,\cdot\,+y)+f(\,\cdot\,+y)\,{\rm d}\nu(y)\leq Ch
        \]
        for all $h\in [0,h_0]$. Combining these two estimates yields that
        \[
         \big\|S(h)f_\nu- f_\nu\big\|_\infty\leq Ch\quad \text{and}\quad \big\|S(h)(-f_\nu)+ f_\nu\big\|_\infty\leq Ch
        \]
       for all $h\in [0,h_0]$. This shows that $f_\nu\in D_L^s(C,h_0)$.
\end{remark}

\section{Uniqueness}\label{sec:uniqueness}
 In this section, we show that a convex semigroup is uniquely determined on $D(A_\delta)$ through its generator $A_\delta$ if the semigroup is, in addition, monotone and continuous from above. The following is the main result of this paper.

\begin{theorem}\label{thm:uniqueness}
	Let $S$ be a convex monotone $C_0$-semigroup on $X$ which is continuous from above with monotone generator $A_\delta$. Let $y\colon [0,\infty)\to X$ be a continuous function with $y(t)\in D(A_\delta)$ for all $t\ge 0$, and assume that, for all $t\ge 0$ and $(h_n)_{n\in\mathbb{N}}$ in $(0,\infty)$ with $h_n\downarrow 0$, there exists a bounded decreasing sequence $(B_n y(t))_{n\in\mathbb{N}}$ in $X$ such that \[\bigg\|\frac{y(t+h_n)-y(t)}{h_n}-B_n y(t) \bigg\|\to 0\quad\mbox{and}\quad B_n y(t) \downarrow A_\delta y(t).\]
	Then, $y(t)=S(t)x$ for all $t\geq 0$, where $x:=y(0)$. 	
\end{theorem}

\begin{proof}
	Let $t>0$ and $g(s):=S(t-s)y(s)$ for all $s\in [0,t]$.
	Fix  $s\in (0,t)$. For every $h>0$ with $h<t-s$ one has
	\begin{align*}
	\frac{g(s+h)-g(s)}{h}&=\frac{S(t-s-h)y(s+h)-S(t-s)y(s)}{h}\\
	&=\frac{S(t-s-h)y(s+h)-S(t-s-h)y(s)}{h}\\
	&\quad-\frac{S(t-s-h)S(h)y(s)-S(t-s-h)y(s)}{h}.
	\end{align*}
	Let $(h_n)_{n\in\mathbb{N}}$ in $(0,\infty)$ with $h_n\downarrow 0$ and $\mu\in M$. By assumption, for $y:=y(s)\in D(A_\delta)$,
	there exists a bounded decreasing sequence $(B_n y)_{n\in\mathbb{N}}$ in $X$ with
	\begin{equation}\label{eq:Shn}
	\bigg\|\frac{y(s+h_n)-y(s)}{h_n}- B_n y\bigg\|\to 0 \quad\mbox{and}\quad B_n y\downarrow A_\delta y.
	\end{equation}
	We define
	\[
	\nu_n z:=\frac{\mu S(t-s-h_n)(y+h_nz)-\mu S(t-s-h_n)y}{h_n}
	\]
	for all $z\in X_\delta$ and $n\in \N$ with $t-s-h_n>0$, where we take the unique extension of $S$ to $X_\delta$ given by Lemma~\ref{ex:S}. Moreover, let
	\[
	 \nu z:=\limsup_{n\to\infty}\nu_n z\quad \text{for all }x\in X.
	\]
	We first show that
	\begin{equation}\label{eq:bound}
	\nu z\le \inf_{h>0}\frac{\mu S(t-s)(y+h z)-\mu S(t-s)y}{h}\quad \text{for all }z\in X.
	\end{equation}
        Indeed, for every $\varepsilon>0$, there exists some $h_0>0$ and, by Corollary \ref{cor:cont} there exists some $m_0\in\mathbb{N}$ such that
	\begin{align*}
	&\inf_{h>0}\frac{\mu S(t-s)(y+h z)-\mu S(t-s)y}{h}+2\varepsilon \ge \frac{\mu S(t-s)(y+h_0 z)-\mu S(t-s)y}{h_0}+\varepsilon \\
	&\qquad \qquad \ge \frac{\mu S(t-s-h_m)(y+h_0 z)-\mu S(t-s-h_m)y}{h_0}
	\end{align*}
	for all $m\ge m_0$. Hence, for all $n\ge m_0$, which satisfy $h_n\le h_0$, one has
	\begin{align*}
         \inf_{h>0} & \frac{\mu S(t-s)(y+h z)  -\mu S(t-s)y}{h}+2\varepsilon \\
        & \qquad\ge \frac{\mu S(t-s-h_n)(y+h_n z)-\mu S(t-s-h_n)y}{h_n} = \nu_n z,
	\end{align*}
	which shows \eqref{eq:bound} by taking the limit superior as $n\to\infty$ and letting $\varepsilon\downarrow 0$. As a consequence of \eqref{eq:bound}, it follows that $\nu$ is continuous from above (on $X$). Indeed, for every sequence $(z_n)_{n\in\mathbb{N}}$ in $X$ with $z_n\downarrow 0$, one has
	\[
	0\le \inf_{n\in\mathbb{N}} \nu z_n\le\inf_{h>0}\inf_{n\in\mathbb{N}}\frac{\mu S(t-s)(y+h z_n)-\mu S(t-s)y}{h}=0
	\]
	so that $\nu z_n\downarrow 0$. Moreover, by definition, $\nu z=\lim_{n\to \infty}\sup_{k\geq n} \nu_k z$ for all $z\in X$, and therefore $\nu\colon X\to \R$ is convex. By \cite[Lemma 3.9]{denk2018kolmogorov}, $\nu$ uniquely extends to a convex monotone functional $\overline \nu\colon X_\delta\to \R$, which is continuous from above. We next show that
	\begin{equation}\label{eq:limsup}
	\limsup_{n\to\infty}\nu_n B_n y=\overline \nu A_\delta y.
	\end{equation}
	To that end, let $\varepsilon>0$. Then, there exist $n_0,m_0\in\mathbb{N}$ such that
	\[
	\overline \nu A_\delta y +2\varepsilon\ge\overline \nu B_{n_0}y+\varepsilon=\nu B_{n_0}y+\varepsilon\ge \nu_m B_{n_0}y\ge \nu_m B_my
	\]
	for all $m\ge  m_0\vee n_0$, where the last inequality follows by monotonicity of $\nu_m$.\ This shows that \[\overline \nu A_\delta y\ge\limsup_{n\to\infty}\nu_n B_ny.\] Further,
	\begin{align*}
	 \overline \nu A_\delta y&= \inf_{m\in \N} \nu B_m y=\inf_{m\in \N}\inf_{n\in \N}\sup_{k \geq n}\nu_k B_my=\inf_{n\in \N}\inf_{m\in \N}\sup_{k \geq n}\nu_k B_my\\
	 &\leq \inf_{n\in \N} \sup_{k\geq n}\nu_k B_k y=\limsup_{n\to \infty} \nu_nB_ny.
	\end{align*}
	By Lemma \ref{loclipschitz}, there exists some $L>0$ such that
	\[
	\bigg\|\frac{S(t-s-h_n)y(s+h_n) -S(t-s-h_n)\big(y+h_n B_n y\big)}{h_n} \bigg\|\le L\bigg\|\frac{y(s+h_n)-y}{h_n}-B_n y \bigg\|\!\to 0
	\]
	as $n\to\infty$.
	Therefore, we conclude that
	\begin{equation}\label{eq:limsup1}
	\limsup_{n\to\infty} \mu\bigg(\frac{S(t-s-h_n)y(s+h_n)-S(t-s-h_n)y}{h_n}\bigg)=\limsup_{n\to\infty} \nu_n B_n y=\overline\nu A_\delta y.
	\end{equation}
	
	Since $y=y(s)\in D(A_\delta)$, it follows from \eqref{domweak} that there exists a bounded decreasing sequence $(A_n y)_{n\in\mathbb{N}}$ with
	\[
	\bigg\|\frac{S(h_n)y-y}{h_n}- A_n y\bigg\| \to 0\quad\mbox{and}\quad A_n y\downarrow A_\delta y.
	\]
	By the same arguments as before, we get
	\begin{equation}\label{eq:limsup2}
	\limsup_{n\to\infty} \mu\bigg(\frac{S(t-s-h_n)S(h_n)y-S(t-s-h_n)y}{h_n}\bigg)=\limsup_{n\to\infty} \nu_n A_n y=\overline \nu A_\delta y.
	\end{equation}
	Hence, in combination with \eqref{eq:limsup1}, we get
	\begin{align}
	&\limsup_{n\to\infty} \mu\bigg(\frac{S(t-s-h_n)y(s+h_n)-S(t-s-h_n)y(s)}{h_n}\bigg)\nonumber \\
	&\qquad =\limsup_{n\to\infty} \mu\bigg(\frac{S(t-s-h_n)S(h_n)y(s)-S(t-s-h_n)y(s)}{h_n}\bigg)\label{eqmu}
	\end{align}
	for every sequence $(h_n)_{n\in\mathbb{N}}$ in $(0,\infty)$ with $h_n\downarrow 0$ and all $\mu\in M$. As a consequence, we conclude that
	\begin{equation}\label{eq:derivativezero}
	\frac{\mu g(s+h_n)-\mu g(s)}{h_n}\to 0
	\end{equation}
	for every sequence $(h_n)_{n\in\mathbb{N}}$ in $(0,\infty)$ with $h_n\downarrow 0$ and all $\mu\in M$. Indeed, by passing to a subsequence $(n_k)_{k\in\mathbb{N}}$, we may assume that
	\[
	\limsup_{n\to\infty}\frac{\mu g(s+h_n)-\mu g(s)}{h_n}= \lim_{k\to\infty}\frac{\mu g(s+h_{n_k})-\mu g(s)}{h_{n_k}}.
	\]
	By passing to another subsequence, which we still denote by $(n_k)_k$, we can further assume that
	\begin{align}
	&\liminf_{k\to\infty} \mu\bigg(\frac{S(t-s-h_{n_k})S(h_{n_k})y(s)-S(t-s-h_{n_k})y(s)}{h_{n_k}}\bigg)\nonumber\\&\qquad =\limsup_{k\to\infty} \mu\bigg(\frac{S(t-s-h_{n_k})S(h_{n_k})y(s)-S(t-s-h_{n_k})y(s)}{h_{n_k}}\bigg).\label{eqmu2}
	\end{align}
	Then, by applying the equality \eqref{eqmu} to the  subsequence $(h_{n_k})_{k\in\mathbb{N}}$ we obtain
	\begin{align*}
	&\limsup_{n\to\infty}\frac{\mu g(s+h_n)-\mu g(s)}{h_n}= \lim_{k\to\infty}\frac{\mu g(s+h_{n_k})-\mu g(s)}{h_{n_k}}\\
	&\qquad\quad\le \limsup_{k\to\infty} \mu\bigg(\frac{S(t-s-h_{n_k})y(s+h_{n_k})-S(t-s-h_{n_k})y(s)}{h_{n_k}}\bigg)\\
	&\qquad\qquad\; -\liminf_{k\to\infty} \mu\bigg(\frac{S(t-s-h_{n_k})S(h_{n_k})y(s)-S(t-s-h_{n_k})y(s)}{h_{n_k}}\bigg)=0,
	\end{align*}
	where the last equality follows from \eqref{eqmu} and \eqref{eqmu2}.
	With similar arguments, we also obtain $\liminf_{n\to\infty}\frac{\mu g(s+h_n)-\mu g(s)}{h_n}\ge 0$, which shows \eqref{eq:derivativezero}.
	
	Since $\mu$ is continuous on $X$, see e.g.~\cite[Theorem 9.6]{aliprantis2006infinite}, it follows by the same arguments as in the proof of \cite[Theorem 3.5]{dkn3} that $s\mapsto \mu g(s)$ is continuous on $[0,t]$. By \cite[Lemma 1.1, Chapter 2]{pazy2012semigroups}, we conclude that the map $s\mapsto \mu g(s)$ is constant on $[0,t]$, since it is continuous and its right derivative vanishes on $[0,t)$. In particular, $\mu y(t)=\mu g(t)=\mu g(0)=\mu S(t)y(0)$ for all $\mu\in M$. This shows that $y(t)= S(t)y(0)$ as $M$ separates the points of $X$.
\end{proof}

\begin{corollary}\label{cor:uniqueness}
	Let $S$ be a convex monotone $C_0$-semigroup on $X$ which is continuous from above with monotone generator $A_\delta$, and let $T$ be a convex $C_0$-semigroup on $X$ with generator $B$ and monotone generator $B_\delta$ such that  $B_\delta\subset A_\delta$.
	If $\overline{D(B)}=X$, then $S(t)=T(t)$ for
	all $t\geq 0$.	
\end{corollary}
\begin{proof}
	For every $x\in D(B)$, the mapping $y\colon [0,\infty)\to X$, $y(t):=T(t)x$ satisfies the assumptions of Theorem \ref{thm:uniqueness}. Indeed, $y(0)=x$ by definition, $t\mapsto y(t)$ is continuous by Corollary \ref{cor:cont}, and $y(t)\in D(B_\delta)\subset  D(A_\delta)$ by Theorem \ref{generatormain2} with
	\[\bigg\|\tfrac{y(t+h_n)-y(t)}{h_n}-B_n y(t)\bigg\|\to 0\quad\mbox{and}\quad B_n y(t)\downarrow B_\delta y(t)= A_\delta y(t)\]
	where $B_n y(t):=\tfrac{T(t)(x+h_nB x)-T(t)x}{h_n}$ for all $n\in \N$.
	Hence, by Theorem \ref{thm:uniqueness}, it follows that $T(t)x=y(t)=S(t)x$ for all $t\ge 0$. Since, by Lemma \ref{cor:Lip}, the bounded convex functions $T(t)$ and $S(t)$ are continuous, and $\overline{D(B)}=X$, it holds
	$S(t)=T(t)$ for all $t\ge 0$.
\end{proof}	

\section{Examples}\label{sec:examples}
\subsection{The uncertain shift semigroup}\label{sec:ushift}
Let $G$ be a convex set endowed with a metric $d\colon G\times G\to [0,\infty)$. We assume that, for every $x,y\in G$ and $\lambda\in(0,1)$, there exists some $\lambda(x,y)\in G$ such that $d(x,\lambda(x,y))=\lambda d(x,y)$ and $d(\lambda(x,y),y)=(1-\lambda) d(x,y)$. The space of all bounded uniformly continuous functions $f\colon G \to \R$ is denoted by $\BUC=\BUC(G)$  and endowed with
the supremum norm $\|f\|_\infty:=\sup_{x\in G}|f(x)|$. Notice that $\BUC$ is a Riesz subspace of the Dedekind $\sigma$-complete Riesz space $\mathcal{L}^\infty$ of all bounded Borel measurable functions $f\colon G\to \R$. On $\mathcal{L}^\infty$ we consider the partial order $f\le g$ whenever $f(x)\le g(x)$ for all $x\in G$.

The \emph{uncertain shift semigroup} $S$ on $\BUC$ is defined by
\[
\big(S(t)f\big)(x):=\sup_{d(x,y)\le t} f(y)\quad \text{for all }f\in \BUC,\; x\in G\text{ and }t\geq 0.
\]

\begin{lemma}\label{lem:robshift}
	$S$ is a sublinear monotone $C_0$-semigroup on $\BUC$. Moreover, $$D_L=D_L^s=\Lip\nolimits_b,$$ where $\Lip_b=\Lip_b(G)$ is the space of all bounded Lipschitz continuous functions $G\to \R$.
\end{lemma}
\begin{proof}
	We first show that $S(t)\colon \BUC\to \BUC$ is well-defined and bounded. To this end, fix $f\in \BUC$. Since
	\[
	|S(t)f(x)|\le \sup_{d(x,y)\le t}|f(y)|\le \|f\|_\infty\quad\mbox{for all }x\in G,
	\]
	it follows that $\|S(t)f\|_\infty\le \|f\|_\infty$.
	Fix $\varepsilon >0$ and $\delta>0$ such that
	$|f(x)-f(y)|\leq \varepsilon$
	for all $x,y\in G$ with $d(x,y)\leq \de$. Let $x,y\in G$ with $d(x,y)\leq \de$ and $z\in G$ with $d(x,z)\leq t$. Then, for $\la:=\frac{t}{t+\de}$, one has
	\[
	d\big(y,\la(y,z)\big)=\la d(y,z)\leq \la(t+\delta) = t
	\]
	and
	\[
	d\big(z,\la(y,z)\big)= (1-\la) d(y,z)\leq (1-\la) (t+\de)= \de
	\]
	Hence,
	\[
	f(z)-\big(S(t)f\big)(y)\leq f(z)-f\big(\la(y,z)\big)\leq \ep.
	\]
	Taking the supremum over all $z\in G$ with $d(x,z)\leq t$, it follows that
	\[
	\big(S(t)f\big)(x)-\big(S(t)f\big)(y)\leq \ep.
	\]
	By a symmetry argument, we obtain that $|S(t)f(x)-S(t)f(y)|\leq \ep$, showing that $S(t)f$ is uniformly continuous with the same modulus of continuity as $f$. We thus have shown that $S(t)\colon \BUC\to \BUC$ is well-defined and bounded. By definition, each $S(t)$ is sublinear and monotone, and $S(0)f=f$ for all $f\in \BUC$. Moreover, for $t\leq \de$, 
	\[
	\big|\big(S(t)f\big)(x)-f(x)\big|\leq\sup_{d(x,y)\leq t}|f(y)-f(x)|\leq \ep
	\]
	for all $x\in G$. Hence, $\|S(t)f-f\|_\infty\leq \ep$ for all $t\leq \de$, which shows that $S$ is strongly continuous. It remains to show that $S$ satisfies the semigroup property. 
	Let $s,t\geq 0$. Further, let $x\in G$ and $z\in G$ with $d(x,z)\leq s+t$. Then, for $\la:=\frac{t}{s+t}$, it holds
	\[
	d\big(z,\la(x,z)\big)=(1-\la) d(x,z)\leq s
	\]
	and
	\[
	d\big(x,\la(x,z)\big)=\la d(x,z)\leq t.
	\]
	Hence,
	\[
	f(z)\leq \sup_{d(\la(x,z),y)\leq s} f(y)=\big(S(s)f\big)\big(\la(x,z)\big)\leq \sup_{d(x,y)\leq t} \big(S(s)f\big)(y)=\big(S(t)S(s)f\big)(x).
	\]
	Taking the supremum over all $z\in G$ with $d(x,z)\leq s+t$, it follows that
	\[
	\big(S(s+t)f\big)(x)\leq \big(S(t)S(s)f\big)(x) .
	\]
	Now, let $z\in G$ with $d(x,z)\leq t$. Then, there exists a sequence $(z_n)_{n\in\mathbb{N}}$ in $G$ with $d(z,z_n)\leq s$ and $f(z_n)\to \big(S(s)f\big)(z)$.
	Then,
	\[
	\big(S(s)f\big)(z)=\lim_{n\to \infty} f(z_n)\leq \sup_{d(x,y)\leq s+t}f(y)=\big(S(s+t)f\big)(x).
	\]
	Taking the supremum over all $z\in G$ with $d(x,z)\leq t$, yields that
	\[
	\big(S(t)S(s)f\big)(x)\leq \big(S(s+t)f\big)(x).
	\]
	Altogether, we have shown that $S$ is a sublinear monotone $C_0$-semigroup on $\BUC$.
	
	Now, let $f\in D_L$. Then, there exist $h_0>0$ and $C\geq 0$ such that $\|S(h)f-f\|_\infty\leq Ch$ for all $h\in [0,h_0]$. Hence, for all $x,y\in G$ with $d(x,y)=:h\leq h_0$,
	\[
	 f(x)-f(y)\leq \big(S(h)f\big)(y)-f(y)\quad\text{and}\quad f(y)-f(x)\leq \big(S(h)f\big)(x)-f(x).
	\]
	This implies that $|f(x)-f(y)|\leq \|S(h)f-f\|_\infty\leq Ch=Cd(x,y)$. Since $f\in \BUC$ is bounded, it follows that $f\in \Lip_b$. On the other hand, 
	if $f\in \Lip_b\subset \BUC$ with Lipschitz constant $C>0$, it follows that
	\[
	 \|\big(S(h)f\big)(x)-f(x)\|\leq \sup_{d(x,y)\leq h}|f(y)-f(x)|\leq Cd(x,y)\leq Ch
	\]
        for all $x\in G$ and $h\geq 0$. Therefore, $f\in D_L$. Since $-f\in \Lip_b$ for all $f\in \Lip_b$, it follows that $\Lip_b\subset D_L^s$. Since, by definition, $D_L^s\subset D_L$, the assertion follows.
\end{proof}	

We now specialize on the case, where $G=\mathbb{R}$ endowed with the Euclidean distance $d(x,y)=|x-y|$. In this case, the uncertain shift semigroup is given by
\[
\big(S(t)f\big)(x)=\sup_{|y|\le t} f(x+y)
\]
for all $x\in\R$ and $t\in[0,\infty)$. By Lemma \ref{lem:robshift}, it follows that $S$ is a sublinear monotone $C_0$-semigroup on $\BUC$. 
In addition, by Dini's lemma, it is continuous from above. Denote by $A_\delta\colon D(A_\delta)\subset \BUC\to \BUC_\delta$ the monotone generator of $S$. 
Notice that $\BUC_\delta$ is the space of all bounded upper semicontinuous functions $\mathbb{R}\to \R$. Moreover, by Lemma \ref{lem:robshift}, we have that $D_L=D_L^s=W^{1,\infty}$. Recall that the space of all Lipschitz continuous functions coincides with the space $W^{1,\infty}$ of all functions with weak derivative $f^\prime\in L^\infty$ (w.r.t.~the Lebesgue measure). As usual, we denote by $\BUC^1$ the space of all $f\in\BUC$ which are differentiable with $f^\prime\in \BUC$. From a PDE point of view, one might consider $\BUC^1$ to be the canonical choice for the
domain of the generator of $S$. However, the following example shows that this does not yield an
m-accretive operator.

\begin{example}
  \label{1.5}
  Let $X=\BUC$ and $B\colon D(B)\to X$ with $Bf := |f'|$ for $f\in D(B) := \BUC^1$. Then $B$ is accretive, i.e., for some (equivalently, for any) $h>0$,  $1+hB$ is injective and
  \[ \big\| (1+hB)^{-1} g_1 - (1+hB)^{-1} g_2 \big\| \le \|g_1-g_2\|\quad \text{for all }g_1,g_2\in R(1+hB),\]
  cf.\ \cite[Proposition 3.1]{Barbu10} and \cite[Formula (8)]{Crandall-Liggett71}. To see this,
  let $f_1,f_2\in D(B)$ and $h>0$. We set $g:= f_1-f_2$ and choose a sequence $(x_k)_{k\in\mathbb{N}}$ in $\R$ with
  $|g(x_k)|\to \|g\|_\infty$ as $k\to\infty$. If $(x_k)_{k\in\mathbb{N}}$ has a finite accumulation point $x_0$, then
  we have $|g(x_0)|=\|g\|_\infty$, and the function $g$ has a local extremum at $x_0$. Consequently, $g'(x_0)=0$ and
  therefore $f_1'(x_0) = f_2'(x_0)$. We obtain
  \begin{align*}
    \big\| f_1 -f_2 + & h  (|f_1'|-|f_2'|)\big\|_\infty \ge \big| f_1(x_0) - f_2(x_0) +  h  (|f_1'(x_0)|
    - |f_2'(x_0)|\big| \\
    & = |f_1(x_0)-f_2(x_0)| = |g(x_0)| = \|g\|_\infty = \|f_1-f_2\|_\infty.
  \end{align*}
  If $(x_k)_{k\in\N}$ has no finite accumulation point, we may w.l.o.g.\ assume that $x_k\to \infty$ as $k\to\infty$.
  Moreover, taking a subsequence we may also assume that $g(x_k)\to \pm \|g\|_\infty$ as $k\to\infty$.
  Again, w.l.o.g.\ let $g(x_k)\to \|g\|_\infty$ as $k\to\infty$. Let $\ep>0$, and choose $k_0\in\N$ with
  \[ \|g\|_\infty - \tilde\ep \le g(x_k) \le \|g\|_\infty \quad \text{for all }k\ge k_0,\]
  where we have set $\tilde \ep := \min\{\tfrac{\ep}{2}, \tfrac{\ep}{2h} \}$. Let $\ell_0>k_0$ with
  $x_{\ell_0}\ge x_{k_0}+1$.   As $g\in\BUC^1$, there exists some
  $y_0\in (x_{k_0}, x_{\ell_0})$ with
  \[ \tilde\ep \ge |g(x_{\ell_0}) - g(x_{k_0})| = |g'(y_0)|\, |x_{\ell_0}-x_{k_0}|  \ge |g'(y_0)|.\]
  We obtain
  \begin{align*}
    \big\|f_1-f_2 + & h  (|f_1'| - |f_2'|)\big\|_\infty \ge \big| f_1(y_0) - f_2(y_0) +  h  (|f_1'(y_0)|
    - |f_2'(y_0)|)\big| \\
    & \ge |f_1(y_0)-f_2(y_0)| -  h  \big| |f_1'(y_0)|    - |f_2'(y_0)|\big|\\
    & \ge |f_1(y_0)-f_2(y_0)| -  h   |  f_1'(y_0)     -  f_2'(y_0) |\\
    & =  |g(y_0)| -  h  |g'(y_0)| \ge \|g\|_\infty -\tfrac\ep 2- h\,\tfrac\ep {2h} = \|f_1-f_2\|_\infty - \ep.
  \end{align*}
  As $\ep>0$ was arbitrary, we see that also in this case the inequality
  \[ \big\|f_1-f_2 +  h  (|f_1'| - |f_2'|)\big\|_\infty \ge \|f_1-f_2\|_\infty \]
  holds, which shows that $B$ is accretive.

  However, the operator $B$ is not m-accretive, i.e., the operator $1+hB$ is not surjective. For this, let $h>0$, and
  set $u(x) := (1-|x|)\boldsymbol 1_{[-1,1]}(x)$ for $x\in\R$. Assume that there exists some $f\in D(B)$ with
  \begin{equation}\label{ex.1.5-1-1}
  f(x) + h|f'(x)| = u(x) \quad \text{for }x\in\R.
  \end{equation}
  As $u$ is an even function, we see that the function $\bar f$ defined by $\bar f(x) := f(-x)$ is also a solution
  of \eqref{ex.1.5-1-1}. As $B$ is accretive, the operator $1+hB$ is injective, which shows that
  $\bar f=f$, i.e., the solution $f$ is an even function, too. As $f\in \BUC^1$, we get $f'(0)=0$ and
  therefore $f(0)=u(0)=1$. Now, the differentiability of $f$ leads to a contradiction to $f(x)\le u(x)$ for all $x\in \R$, which holds
  by \eqref{ex.1.5-1-1}.
\end{example}

\begin{proposition}\label{prop:firstorder}
	Let $G=\R$. Then, $\BUC^1\subset D(A)\subset D(A_\delta)\subset D_L=D_L^s=W^{1,\infty}$.
	In particular, $S(t)f\in W^{1,\infty}$ for every $f\in W^{1,\infty}$ and all $t\ge 0$.	Further, for $f\in D(A_\delta)$, one has $A_\delta f=|f'|$ almost everywhere.
\end{proposition}
\begin{proof}
	If $f\in \BUC^1$, it follows from Taylor's theorem that
	\[
	\bigg\|\frac{S(h)f-f}{h}-|f^\prime| \bigg\|_\infty\to 0\quad\mbox{as }h\downarrow 0.
	\]	
	Hence, by Lemma \ref{lem:domain} and Lemma \ref{lem:robshift}, \[\BUC^1\subset D(A)\subset D(A_\delta)\subset D_L=D_L^s=W^{1,\infty}.\]
	In particular, $W^{1,\infty}$ is invariant under the uncertain shift semigroup by Theorem \ref{generatormain2}.
	
	Let $f\in W^{1,\infty}$.
	By Rademacher's theorem the function $f$ is differentiable almost everywhere. If $f$ is differentiable at $x$, then
	\begin{align*}
	\lim_{h\downarrow 0}\frac{\big(S(h)f\big)(x)-f(x)}{h}&=\lim_{h\downarrow 0}\sup_{|y|\le h}\frac{f(x+y)-f(x)}{h}\\&=\lim_{h\downarrow 0}\sup_{|y|=h}\frac{f(x+y)-f(y)}{h}=|f^\prime(x)|.
	\end{align*}
	Since, for $f\in D(A_\delta)$, one has
	\[
	\big(A_\delta f\big)(x)=\lim_{h\downarrow 0}\frac{\big(S(h)f\big)(x)-f(x)}{h}
	\]
	for all $x\in \mathbb{R}^d$, we conclude that $A_\delta f=|f^\prime|$ almost everywhere. Here, $f^\prime$ is understood as the weak derivative in $L^\infty$.
\end{proof}

The following example shows that, in general, $D(A)$ is not invariant under the semigroup $(S(t))_{t\geq 0}$.

\begin{example}\label{4.3}
 Consider the case $G=\R$, and let $f\in \BUC^1$ with
 \[
  f(x)=\begin{cases}
         x^2, & x\in [0,2],\\
         x^4, & x\in [-2,0).
        \end{cases}
 \]
 Then, by Proposition \ref{prop:firstorder}, $S(1)f\in D(A_\delta)\subset W^{1,\infty}$ with $A_\delta S(1)f=\big|\big(S(1)f\big)'\big|$. By definition of $S(1)$,
 \[
    \big(S(1)f\big)(x)=\begin{cases}
         (x+1)^2, & x\in [0,1],\\
         (x-1)^4, & x\in [-1,0),
        \end{cases}
 \]
 which implies that
  \[
  \big(S(1)f\big)'(x)=\begin{cases}
         2(x+1), & x\in (0,1),\\
         4(x-1)^3, & x\in (-1,0).
        \end{cases}
 \]
 Therefore, $A_\delta S(1)f=\big|\big(S(1)f\big)'\big|\notin \BUC$ and, in particular, $S(1)f\notin D(A)$.
\end{example}

\subsection{The $G$-expectation}\label{sec:G}
Let $0\le\underline{\sigma}\le\overline{\sigma}$. We consider the $G$-expectation on $\BUC=\BUC(\mathbb{R})$, which corresponds to the sublinear semigroup
\[
\big(S(t)f\big)(x):=\sup_{\sigma\in\Sigma}\mathbb{E}\Big[f\big(x+\int_{0}^t \sigma_s\,dW_s\big)\Big]\quad \text{for }f\in\BUC,\; x\in G\text{ and }t\geq 0,
\]
where $W$ is a Brownian motion on a filtered probability space $(\Omega,\mathcal{F},(\mathcal{F}_t),\mathbb{P})$ and $\Sigma$ denotes the set of all progressively measurable processes with values in $[\underline{\sigma},\overline{\sigma}]$, see e.g.~\cite{MR2754968} and  \cite{peng2010nonlinear} for an overview on $G$-expectations. Note that we do not assume that $\underline{\sigma}>0$, which is a standard assumption in PDE theory for obtaining regularity results in H\"older spaces (cf.\ Lieberman \cite[Chapter XIV]{MR1465184} and Peng \cite[Appendix C, §4]{peng2010nonlinear} for a short survey).
An inspection shows that $S$ is a translation invariant sublinear
$C_0$-semigroup on $\BUC$ which is continuous from above. Moreover, an application of It\^o's formula leads to
\begin{equation}\label{eq.genGsemigroup}
\lim_{h\downarrow 0}\frac{S(h)f-f}{h}=\tfrac{1}{2}\max\big\{\underline{\sigma}^2f^{\prime\prime},\overline{\sigma}^2f^{\prime\prime}\big\}\quad\mbox{for all }f\in \BUC^2,
\end{equation}
 where $\BUC^2$ denotes the set of all $f\in \BUC^1$ with first derivative $f'\in \BUC^1$. In particular, $\BUC^2\subset D(A)$. We next determine the symmetric Lipschitz set.
\begin{lemma}\label{lem:LipsG}
 The symmetric Lipschitz set is given by  $D^s_L= W^{2,\infty}$. 
\end{lemma}
\begin{proof}
First, we show that $D^s_L\subset  W^{2,\infty}$. To that end, fix $f\in D^s_L$. By definition of the symmetric Lipschitz set, there exist $C>0$ and $h_0>0$ such that $f\in D^s_L(C,h_0)$. For every $\delta>0$, we define $f_\delta(x):=\int_{\mathbb{R}} f(x+y)\,\nu_\delta({\rm d}y)$, where $\nu_\delta$ is the normal distribution $\mathcal{N}(0,\delta)$ with mean zero and variance $\delta$.
Then, it holds $f_\delta\in\BUC^2$ for all $\delta>0$, and $\|f_\delta-f\|_\infty\to 0$ as $\delta\downarrow 0$. It follows from Remark~\ref{prop:approx} that
\[S(h) f_\delta- f_\delta \le Ch\quad\mbox{and}\quad -S(h) (-f_\delta)- f_\delta \ge -Ch\] for all $h\in[0,h_0]$
and $\delta>0$. Hence, letting $h\downarrow 0$, we obtain
\[\frac{1}{2}\overline{\sigma}^2f^{\prime\prime}_\delta\le C\quad\mbox{and}\quad \frac{1}{2}\overline{\sigma}^2f^{\prime\prime}_\delta\ge -C.\]
This shows that $\|f^{\prime\prime}_\delta\|_\infty$ is uniformly bounded in $\delta>0$. Hence, there exists a sequence $\delta_n\downarrow 0$ such that $\int_x^y f^{\prime\prime}_{\delta_n}(z)-g(z)\,{\rm d}z\to 0$ for all $x,y\in\mathbb{R}$ with $x<y$ and some $g\in L^\infty$. By the dominated convergence theorem, we obtain
\begin{align*}
	\frac{f(x+h)-f(x)}{h}&=\lim_{n\to\infty}\bigg(\frac{f_{\delta_n}(x+h)-f_{\delta_n}(x)}{h}\bigg)\\
	&=\lim_{n\to\infty}\bigg( f_{\delta_n}^\prime(x)+\frac{1}{h}\int_x^{x+h}\int_x^y f_{\delta_n}^{\prime\prime}(z)\,{\rm d}z\,{\rm d}y\bigg)\\
	&=\lim_{n\to\infty}  f^\prime_{\delta_n}(x)+\frac{1}{h}\int_x^{x+h}\int_x^y g(z)\,{\rm d}z\,{\rm d}y
\end{align*}
for all $x\in \R$ and $h>0$. Since $\lim_{h\downarrow 0}\tfrac{1}{h}\int_x^{x+h}\int_x^y g(z)\,{\rm d}z\,{\rm d}y\to 0$ for all $x\in\mathbb{R}$, we conclude that $f$ is differentiable with $f^\prime(x)=\lim_{n\to\infty} f^\prime_{\delta_n}(x)$ and second weak derivative $f^{\prime\prime}=g$.
For the special choice $h=1$, we observe that $f^\prime\in L^\infty$. This shows that $f\in W^{2,\infty}$. 

Second, we show that  $W^{2,\infty}\subset D^s_L$. Fix $f\in W^{2,\infty}$ and set $f_n:=\int_{\mathbb{R}} f(\cdot+y)\,\nu_{\delta_n}({\rm d}y)$ for all $n\in\mathbb{N}$.
Then, $f_n\in \BUC^2$ with $\|f^{\prime}_n\|_\infty\le \|f^{\prime}\|_\infty$ and $\|f^{\prime\prime}_n\|_\infty\le \|f^{\prime\prime}\|_\infty$ for all $n\in\mathbb{N}$.
For all $\sigma\in \Sigma$, $x\in\mathbb{R}$, and $n\in\mathbb{N}$, it follows from  It\^o's formula that
\[
\mathbb{E}\Bigg[f_n\bigg(x+\int_{0}^h \sigma_s\,dW_s\bigg)\Bigg]-f_n(x)\le\tfrac{1}{2}\|f^{\prime\prime}\|_\infty\overline{\sigma}^2 h.
\]
Since $f_n\to f$ pointwise and $\|f_n\|_\infty\le\|f\|_{\infty}$, we obtain from the dominated convergence theorem that 
\[
\mathbb{E}\Bigg[f\bigg(x+\int_{0}^h \sigma_s\,dW_s\bigg)\Bigg]-f(x)\le\tfrac{1}{2}\|f^{\prime\prime}\|_\infty\overline{\sigma}^2 h.
\]
Taking the supremum over all $\sigma\in \Sigma$, replacing $f$ by $-f$, and using the subadditivity of $S(h)$ yield
\[
S(h)f-f\le \tfrac{1}{2}\|f^{\prime\prime}\|_\infty\overline{\sigma}^2 h\quad\mbox{and}\quad -S(h)f+f\le  S(h)(-f)+f\le \tfrac{1}{2}\|f^{\prime\prime}\|_\infty\overline{\sigma}^2 h.
\]
Combining the previous two estimates implies that $f\in D_L^s$. 
\end{proof}	

\mnen{The previous result can be seen as a regularity result, ensuring that $W^{2,\infty}$-regularity of the initial value is preserved under the semigroup. The following proposition shows that the equality \eqref{eq.genGsemigroup} extends from $\BUC^2$ to $W^{2,\infty}$ in a pointwise sense almost everywhere.} 

\begin{proposition}
Let $f\in W^{2,\infty}$. Then, $S(t)f\in W^{2,\infty}$ for all $t\ge 0$. Moreover,
\begin{equation*}
\lim_{h\downarrow 0}\frac{S(h)f-f}{h}=\tfrac{1}{2}\max\big\{\underline{\sigma}^2f^{\prime\prime},\overline{\sigma}^2f^{\prime\prime}\big\}\quad\mbox{almost everywhere},
\end{equation*}
where the limit on the left-hand side exists pointwise almost everywhere.
\end{proposition}
\begin{proof}
Fix $f\in W^{2,\infty}$. As an application of Proposition \ref{thm:domsym} and Lemma \ref{lem:LipsG}, we obtain that $D^s_L=W^{2,\infty}$ is invariant under the semigroup $S$, i.e., $S(t)f\in W^{2,\infty}$ for all $t\ge 0$.

For $\sigma\in \Sigma$, we consider the stochastic integral $X_{\sigma,t}:=\int_0^t \sigma_s\,dW_s$ for all $t\ge 0$. Since $f\in W^{2,\infty}$, it holds $f^\prime\in \Lip_b$ and, by Rademacher's theorem, the function $f^\prime$ is differentiable almost everywhere, and the pointwise derivative coincides with the weak derivative $f''\in L^\infty$ almost everywhere. 
Suppose that $f^{\prime}$ is differentiable at $x\in \R$  with $|f''(x)|\leq \|f''\|_{\infty}$, so that $f(x+\xi)=f(x)+f^\prime(x)\xi+\tfrac{1}{2}f^{\prime\prime}(x)\xi^2+o(\xi^2)$ as $|\xi|\to 0$.\footnote{Indeed, since $f^{\prime}$ is differentiable at $x$, it holds $\sup_{0\le\tau\le 1} |f^\prime(x+\tau \xi)-f^\prime(x)-\tau \xi f^{\prime\prime}(x)|=o(\xi)$. Integrating w.r.t.~$\tau$ yields $\big|\tfrac{f(x+\xi)-f(x)}{\xi}-f^{\prime}(x)-\tfrac{1}{2}\xi f^{\prime\prime}(x)\big|\le \int_0^1 |f^\prime(x+\tau \xi)-f^\prime(x)-\tau \xi f^{\prime\prime}(x)|\,{\rm d}\tau=o(\xi)$.} Then, we obtain for every $h>0$,
\begin{align*}
	&\Bigg| \frac{\big(S(h)f\big)(x)-f(x)}{h}-\frac{1}{2}\max\Big\{\underline{\sigma}^2f^{\prime\prime}(x),\overline{\sigma}^2f^{\prime\prime}(x)\Big\}\Bigg| \\
	&\qquad = 	\Bigg| \frac{\big(S(h)f\big)(x)-f(x)}{h}-\sup_{\sigma\in \Sigma}\frac{1}{2}f^{\prime\prime}(x)\mathbb{E}\bigg[\frac{X^2_{\sigma,h}}{h}\bigg]\Bigg| \\
	&\qquad \le \sup_{\sigma\in \Sigma} \Bigg| \frac{1}{h}\mathbb{E}\Big[f\big(x+X_{\sigma,h}\big)-f(x)-\frac{1}{2}f^{\prime\prime}(x) X^2_{\sigma,h} \Big] \Bigg|\\
	&\qquad= \sup_{\sigma\in \Sigma} \bigg| \frac{1}{h}\mathbb{E}\big[R(X_{\sigma,h}) \big] \bigg|\quad \mbox{with } R(\xi):=f(x+\xi)-f(x)-f^\prime(x)\xi-\tfrac{1}{2}f^{\prime\prime}(x)\xi^2.
\end{align*}
Note that $(X_{\sigma, t})_{t\geq 0}$ is a martingale, so that
$\mathbb{E}[X_{\sigma,t}]=0$ for all $t\geq 0$. Since the first derivative $f'$ of $f$ is Lipschitz with Lipschitz constant $\|f''\|_\infty$ and $|f''(x)|\leq \|f''\|_{\infty}$, it follows that $$|R(\xi)|\leq \int_0^1 \big|\big(f^\prime(x+\tau \xi)-f^\prime(x)\big)\xi-\tau \xi^2 f^{\prime\prime}(x)\big|\,{\rm d}\tau \le 2\|f^{\prime\prime}\|_\infty \xi^2\quad\text{for all }\xi\in \R,$$ and the Burkholder-Davis-Gundy inequality implies that there exists a constant $C>0$, such that
\begin{align*}
	\sup_{\sigma\in \Sigma} \bigg|\frac{1}{h}\mathbb{E}\big[R(X_{\sigma,h})\boldsymbol 1_{\{|X_{\sigma,h}|\ge\delta\}}\big]\bigg|&\le \sup_{\sigma\in \Sigma}\frac{1}{h}\mathbb{E}\big[R(X_{\sigma,h})^2\big]^{1/2}\cdot \mathbb{P}\big(|X_{\sigma,h}|\ge\delta\big)^{1/2}\\
	&\le\frac{1}{h}\cdot C \overline{\sigma}^2 h \cdot \frac{\overline{\sigma}\sqrt{h}}{\delta}\le \frac{C \overline{\sigma}^3\sqrt{h}}{\delta}\quad \text{for all }\delta >0.
\end{align*}
Let $\ep>0$. Since $R(\xi)=o(\xi^2)$, there exists some $\delta>0$ such that $ R(\xi)\leq \ep \xi^2$ for $|\xi|< \delta$. Hence,
\[
\sup_{\sigma\in \Sigma} \bigg|\frac{1}{h}\mathbb{E}\big[R(X_{\sigma,h})\big]\bigg|\le \frac{C \overline{\sigma}^3\sqrt{h}}{\delta} + \sup_{\sigma\in \Sigma} \bigg|\frac{1}{h}\mathbb{E}\big[R(X_{\sigma,h})\boldsymbol 1_{\{|X_{\sigma,h}|<\delta\}}\big]\bigg|\le \frac{C \overline{\sigma}^3\sqrt{h}}{\delta} + \ep \overline \sigma ^2.
\]
Letting $h\downarrow 0$, this shows that $\tfrac{(S(h)f)(x)-f(x)}{h}\to \tfrac{1}{2}\max\{\underline{\sigma}^2f^{\prime\prime}(x),\overline{\sigma}^2f^{\prime\prime}(x)\}$. The proof is complete. 
\end{proof}

\begin{appendix}

\section{Some auxiliary results}\label{sec:convexsemigroup}
In this section, we list some basic properties for convex operators and semigroups, which can be found, for example, in \cite{dkn3}.

\begin{lemma}[{\cite[Corollary A.4]{dkn3}}]\label{cor:Lip}
Let $S\colon X\to X$ be a bounded and convex operator. Then, $S$ is Lipschitz on bounded subsets, i.e., for every $r>0$, there exists some $L>0$ such that $\|Sx-Sy\|\leq L \|x-y\|$ for all $x,y\in B(0,r)$.
\end{lemma}

For the remainder of this subsection, let $S$ be a convex $C_0$-semigroup on $X$.
\begin{lemma}[{\cite[Corollary 2.4]{dkn3}}]\label{loclipschitz}
 Let $T>0$ and $x_0\in X$. Then, there exist $L\geq 0$ and $r>0$ such that
 \[
  \sup_{t\in [0,T]}\|S(t)y-S(t)z\|\leq L\|y-z\|
 \]
 for all $y,z\in B(x_0,r)$.
\end{lemma}

\begin{corollary}[{\cite[Corollary 2.5]{dkn3}}]\label{cor:cont}
 The map $[0,\infty)\to X,\; t\mapsto S(t)x$ is continuous for all $x\in X$.
\end{corollary}


\begin{proposition}[{\cite[Proposition 2.7]{dkn3}}]\label{domainlip}
Let $x\in X$ with
\[
\sup_{h\in (0,h_0]}\bigg\|\frac{S(h)x-x}{h}\bigg\|<\infty\quad  \text{for some }h_0>0.
\]
Then, the map $[0,\infty)\to X$, $t\mapsto S(t)x$ is locally Lipschitz continuous, i.e., for every $T>0$, there exists some $L_T\geq 0$ such that $\|S(t)x-S(s)x\|\leq L_T|t-s|$ for all $s,t\in [0,T]$.
\end{proposition}

	\section{Directional derivatives of convex operators}\label{append:direcder}
	In this section, we provide some results on directional derivatives of convex operators. 
	
	\begin{lemma}\label{lemma:welldef}
		Let $(x_n)_{n\in\mathbb{N}}$ be a sequence in $X$. If
		$(y_n)_{n\in\mathbb{N}}$ and $(z_n)_{n\in\mathbb{N}}$ are decreasing sequences in $X$
		which are bounded from below such that
		$\|x_n-y_n\|\to 0$ and $\|x_n-z_n\|\to 0$, then
		$\inf_{n\in\mathbb{N}} y_n=\inf_{n\in\mathbb{N}} z_n$.
	\end{lemma}
	\begin{proof}
		Fix $\mu\in M$. Since $\mu$ is continuous on $X$, see e.g.~\cite[Theorem 9.6]{aliprantis2006infinite}, one has \[\mu(y_n-z_n)=\mu(y_n-x_n)+\mu(x_n-z_n)\to 0,\] which shows that
		\[
		\mu \Big(\inf_{n\in\mathbb{N}} y_n\Big)=\lim_{n\to\infty}\mu  y_n+\lim_{n\to\infty}\mu(z_n-y_n)=\lim_{n\to\infty}\mu z_n=\mu\Big(\inf_{n\in\mathbb{N}}  z_n\Big).
		\]
		Since $\inf_{n\in\mathbb{N}} y_n,\inf_{n\in\mathbb{N}}  z_n\in X_\delta$ and $M$ separates the points of $X_\delta$, it follows that
		$\inf_{n\in\mathbb{N}} y_n=\inf_{n\in\mathbb{N}}  z_n$.
	\end{proof}

	\begin{lemma}\label{ex:S}
		Let $S\colon X\to X$ be a convex monotone operator which is continuous from above. Then, it has a unique monotone convex extension $S\colon X_\delta\to X_\delta$ which is continuous from above.
	\end{lemma}
	\begin{proof}
		For each $\mu\in M$,
		the  convex monotone functional $\mu S\colon X\to\mathbb{R}$ is continuous from above. Thus, by \cite[Lemma 3.9]{denk2018kolmogorov}, it has a unique extension to a convex monotone functional $\mu S\colon X_\delta\to\mathbb{R}$ which is continuous from above.
		
		Fix $x\in X_\delta$. For $(x_n)_{n\in\mathbb{N}}$ and $(y_n)_{n\in\mathbb{N}}$ in $X$ with $x_n\downarrow x$ and $y_n\downarrow x$, one has
		\[
		\mu\Big(\inf_{n\in\mathbb{N}} Sx_n\Big)=\inf_{n\in\mathbb{N}} \mu S x_n=\mu S\Big(\inf_{n\in\mathbb{N}} x_n\Big)=\mu S\Big(\inf_{n\in\mathbb{N}} y_n\Big)=\inf_{n\in\mathbb{N}} \mu S y_n=\mu\Big(\inf_{n\in\mathbb{N}} Sy_n\Big).
		\]
		Hence, $Sx:=\inf_{n\in\mathbb{N}} Sx_n$ is well defined as $M$ separates the points of $X_\delta$. Then, $S$ is convex and continuous from above as
		\[
		\mu\Big(\inf_{n\in\mathbb{N}} Sx_n\Big)=\inf_{n\in\mathbb{N}}\mu Sx_n=\mu Sx
		\]
		for every $(x_n)_{n\in\mathbb{N}}$ in $X_\delta$ with $x_n\downarrow x\in X_\delta$. Moreover, if $\tilde S$ is another extension which is continuous from above, then
		$\tilde Sx=\lim_{n\to \infty} \tilde Sx_n=\lim_{n\to \infty} Sx_n=S x$ for every $(x_n)_{n\in\mathbb{N}}$ in $X$ with $x_n\downarrow x\in X_\delta$, which shows that such an extension is unique.
	\end{proof}

	Let $S\colon X\to X$ be a convex operator. Then, the function
	\[
	\R\setminus\{0\}\to X, \quad h\mapsto\frac{S(x+h y)-S x}{h}
	\]
	is increasing for all $x,y\in X$. Hence, for all $x\in X$, the operators
	\begin{equation}\label{def:Sprime}
	S^\prime_+(x)y:=\inf_{h> 0} \frac{S(x+h y)-Sx}{h}\quad\mbox{and}\quad
	S^\prime_-(x)y:=\sup_{h< 0} \frac{S(x+h y)-Sx}{h}
	\end{equation}
	for $y\in X$ are well-defined with values in $\bar X$ since
	\[
	S^\prime_+(x)y=\inf_{n\in \N} \frac{S(x+h_n y)-Sx}{h_n}\in X_\delta\quad\mbox{and}\quad
	S^\prime_-(x)y=\sup_{n\in \N} \frac{Sx-S(x-h_n y)}{h_n}\in -X_\delta
	\]
	for every sequence $(h_n)_{n\in\mathbb{N}}$ in $(0,\infty)$ with $h_n\to  0$. The following properties follow directly from the definition.
	
	\begin{remark}
		For every $x,y\in X$, it holds
		\begin{itemize}
			\item[(i)] $S^\prime_-(x)y=-S^\prime_+(x)(-y)$,
			\item[(ii)] $S^\prime_-(x)y\le S^\prime_+(x)y$,
			\item[(iii)] $S^\prime_+(x)y=S^\prime_-(x)y=Sy$, if $S$ is linear.
		\end{itemize}
	\end{remark}
	If $S\colon X\to X$ is a convex monotone operator which is continuous from above, then by Lemma \ref{ex:S}, it has a unique convex monotone extension $S\colon X_\delta\to X_\delta$ which is continuous from above. Therefore, $S(x+hy)\in X_\delta$
	for all $y\in X_\delta$ and $h>0$. Hence, $S^\prime_+(x)$ extends to
	\[
	S^\prime_+(x)\colon X_\delta\to X_\delta,\quad y\mapsto \inf_{h>0} \frac{S(x+h y)-Sx}{h}
	\]
	for all $x\in X$.
	
	\begin{lemma}\label{lem:Sprime0}
		Let $S\colon X\to X$ be a convex monotone operator which is continuous from above. For every $x\in X$, the mapping $S^\prime_+(x)$ has the following properties:
		\begin{itemize}
			\item[(i)] $S^\prime_+(x)y\leq S_x y$ for all $y\in X_\delta$,
			\item[(ii)] $S^\prime_+(x)\colon X_\delta\to X_\delta$ is convex and positive homogeneous,
			\item[(iii)] $S^\prime_+(x)$ is continuous from above,
			\item[(iv)]  $\tfrac{S(x+h_n y_n)-Sx}{h_n}\downarrow S^\prime_+(x)y
			$ for all sequences $(h_n)_{n\in\mathbb{N}}$ in $(0,\infty)$ and $(y_n)_{n\in\mathbb{N}}$ in $X_\delta$ which satisfy  $h_n\downarrow 0$ and $y_n\downarrow y\in X_\delta$.
			
		\end{itemize}	
	\end{lemma}
	
	\begin{proof}
		(i) For every $y\in X_\delta$, one has $S^\prime_+(x)y\le S(x+y)-S(x)=S_x(y)$.
		
		(ii) For $\varepsilon>0$, $\mu\in M$, and $\lambda\in[0,1]$, there exists some $h>0$ such that
		\begin{align*}
		&\mu\big(\lambda S^\prime_+(x)y_1 +(1-\lambda) S^\prime_+(x)y_2\big)+\varepsilon\\
		&\qquad\quad \ge  \lambda \frac{\mu S(x+h y_1)-\mu S(x)}{h} + (1-\lambda) \frac{\mu S(x+h y_2)-\mu S(x)}{h}\\
		&\qquad\quad \ge \frac{\mu S\big(x+h(\lambda y_1+(1-\lambda)y_2)\big)-\mu S(x)}{h}
		\ge \mu S^\prime_+(x)\big(\lambda y_1+(1-\lambda)y_2\big).
		\end{align*}
		This shows that $S^\prime_+(x)$ is convex on $X_\delta$. Moreover, for $\la>0$ and $y\in X_\de$, it holds
		\[
		S^\prime_+(x)(\la y)=\inf_{h>0} \frac{S(x+\la h y)-Sx}{h}=\la \inf_{h>0}\bigg(\frac{S(x+\la h y)-Sx}{\la h}\bigg)=\la S^\prime_+(x)y.
		\]
		
		(iii) For every $y_n\downarrow y$,
		\[
		\inf_{n\in\mathbb{N}} S^\prime_+(x)y_n=\inf_{h>0}\inf_{n\in\mathbb{N}} \frac{S(x+h y_n)-S(x)}{h}=\inf_{h>0} \frac{S(x+h y)-S(x)}{h}=S^\prime_+(x)y.
		\]
		
		(iv) Fix $\varepsilon>0$, and $\mu\in M$. By definition of $S^\prime_+$ and continuity from above of $S$, there exist $n_0,m_0\in\mathbb{N}$ such that
		\begin{align*}
		\mu S^\prime_+(x)y+2\varepsilon&\ge\frac{\mu S(x+h_{n_0}y)-\mu Sx}{h_{n_0}}+\varepsilon \ge  \frac{\mu S(x+h_{n_0}y_{m_0})-\mu Sx}{h_{n_0}}\\
		&\ge \frac{\mu S(x+h_{n_1}y_{n_1})-\mu Sx}{h_{n_1}}
		\end{align*}
		for $n_1:= n_0\vee m_0$. This shows that $\frac{S(x+h_n y_n)-Sx}{h_n}\downarrow S^\prime_+(x)y$. The proof is complete.
	\end{proof}	
	
\end{appendix}

%

\end{document}